\documentclass[oneside]{amsart}
\usepackage{enumitem}
\usepackage{latexsym, bbm, amssymb, amsmath, wasysym}
\usepackage[abbr,dcucite]{harvard}
\usepackage{setspace}
\usepackage{ifpdf}
\usepackage[capposition=top]{floatrow}

\usepackage{tikz}
\usetikzlibrary{patterns,snakes,shapes}

\usepackage{accents}

\makeatletter
\def\subsection{\@startsection{subsection}{1}%
  \z@{.5\linespacing\@plus.7\linespacing}{.3\linespacing}%
  {\normalfont\scshape}}
\newcommand{\eqnum}{\refstepcounter{equation}\textup{\tagform@{\theequation}}}
\makeatother

\theoremstyle{plain}
\newtheorem{theorem}{Theorem}

\newtheorem{lemma}[theorem]{Lemma}
\newtheorem{proposition}[theorem]{Proposition}
\theoremstyle{definition}

\newtheorem{definition}[theorem]{Definition}

\newtheorem{remark}[theorem]{Remark}

\allowdisplaybreaks[1]

\definecolor{mypurple}{rgb}{.3,0,.5}

\newcommand{\R}{\mathbb{R}}

\newcommand{\N}{\mathbb{N}}
\newcommand{\Z}{\mathbb{Z}}

\renewcommand{\P}{\mathbb{P}}

\newcommand{\E}{\mathbb{E}}

\newcommand{\M}{\mathcal{M}}

\newcommand{\U}{\mathcal{U}}
\newcommand{\A}{\mathcal{A}}

\newcommand{\WW}{\mathcal{W}}

\newcommand{\T}{\mathcal{T}}
\newcommand{\XX}{\mathcal{X}}
\newcommand{\YY}{\mathcal{Y}}

\newcommand{\BB}{\mathcal{B}}
\newcommand{\QQ}{\mathcal{Q}}

\newcommand{\PP}{\mathcal{P}}

\renewcommand{\S}{\mathcal{S}}

\newcommand{\x}{\boldsymbol{x}}

\newcommand{\1}{\mathbbm{1}}

\DeclareMathOperator{\Var}{Var}

\let\oldmarginpar\marginpar
\renewcommand\marginpar[1]{\-\oldmarginpar[\raggedleft\footnotesize #1]%
{\raggedright\footnotesize #1}}

\begin{document}

\title[]%
{Beardwood-Halton-Hammersley Theorem \\ for Stationary Ergodic Sequences: \\ a Counterexample}
\author[]
{Alessandro Arlotto and J. Michael Steele}

\thanks{A. Arlotto: The Fuqua School of Business, Duke University, 100 Fuqua Drive, Durham, NC, 27708.
Email address: \texttt{alessandro.arlotto@duke.edu}}

\thanks{J. M.
Steele:  Department of Statistics, The Wharton School, University of Pennsylvania, 3730 Walnut Street, Philadelphia, PA, 19104.
Email address: \texttt{steele@wharton.upenn.edu}}

\citationmode{full}

\begin{abstract}
        We construct a stationary ergodic process $X_1, X_2, \ldots $
        such that each $X_t$ has
        the uniform distribution on the unit square and the length $L_n$
        of the shortest path through the points $X_1, X_2, \ldots ,X_n$ is \emph{not} asymptotic to
        a constant times the square root of $n$.
        In other words, we show that the \citename{BeaHalHam:PCPS1959} theorem
        does not extend from the case of independent uniformly distributed random variables
        to the case of stationary ergodic sequences with uniform marginal distributions.

        \bigskip


        \bigskip


\end{abstract}
\citationmode{abbr}

\date{
first version: June 30, 2013; this version: August 24, 2015. 
}

\maketitle


\section{Introduction}\label{se:Introduction}

Given a sequence $x_1, x_2, \ldots, x_n, \ldots$ of points in the unit square $[0,1]^2$,
we let $L(x_1,x_2, \ldots, x_n)$ denote the length of the shortest path through the $n$ points $x_1,x_2, \ldots, x_n$; that is,
we let \citationmode{full}
\begin{equation}\label{TSPcost}
L(x_1, x_2, \ldots, x_n)= \min_\sigma   \sum_{t=1}^{n-1} |x_{\sigma(t)} -x_{\sigma(t+1)}| ,
\end{equation}
where the minimum is over all permutations $\sigma: \{1,2,\ldots, n\} \rightarrow \{1,2,\ldots,n\}$ and where $|x-y|$ denotes the
usual Euclidean distance between elements $x$ and $y$ of $[0,1]^2$.
The classical theorem
of \citeasnoun{BeaHalHam:PCPS1959} tells us in the leading case that if $X_1, X_2, \ldots $
is a sequence of independent random variables with the uniform distribution on $[0,1]^2$, then
there is a constant $\beta>0$ such that \citationmode{abbr}
\begin{equation}\label{eq:BHHlimit}
\lim_{n \rightarrow \infty} n^{-1/2} L(X_1, X_2, \ldots, X_n)  = \beta \quad \text{with probability one}.
\end{equation}
The Beardwood, Halton and Hammersley (BHH) constant $\beta$ has been studied extensively, and, although its exact value is still unknown,
sophisticated numerical computations
\cite{AppBixChvaCook:PRINCETON2006} suggest that $\beta \approx 0.714\ldots$. The best available analytical bounds are much rougher;
we only know with certainty that $0.62499 \leq \beta \leq 0.91996$.
\citeaffixed[pp. 497--498.]{Fin:CUP2003}{See, e.g.,}

The BHH theorem is a strong law for independent identically distributed random variables, and,
for reasons detailed below, it is natural
to ask if there is an \emph{analogous ergodic theorem} where one relaxes
the hypotheses to those of the classic ergodic theorem for partial sums. Our main goal is to answer this question;
specifically, we construct a stationary ergodic process with uniform marginals
on $[0,1]^2$ for which the length of the shortest path through $n$ points is
\emph{not} asymptotic to a constant times $n^{1/2}$.

\begin{theorem}[No Ergodic BHH]\label{th:noBHHThm}
There are constants $c_1 < c_2$ and a stationary and ergodic process
$X_1, X_2, \ldots$ such that $X_1$ is uniformly distributed on $[0,1]^2$ and
such that with probability one
\begin{equation}\label{noBBBinequalies}
\liminf_{n \rightarrow \infty} \frac{ L(X_1, X_2, \ldots, X_n) }{ n^{1/2} }
\leq c_1 < c_2 \leq
\limsup_{n \rightarrow \infty} \frac{ L(X_1, X_2, \ldots, X_n) }{ n^{1/2} }.
\end{equation}
\end{theorem}

This theorem is obtained as a corollary of the next theorem where the condition of ergodicity is dropped.
In this case, one can construct processes for which there is a more
explicit control of the expected
minimal path length.

\begin{theorem}[Asymptotics of Expected Path Lengths]\label{th:noBHHThmExpectations}
There is a  stationary
process $X^*_1, X^*_2, \ldots$ such that
$X^*_1$ is uniformly distributed on $[0,1]^2$
and such that
\begin{equation}\label{noBBBinequaliesMean}
\liminf_{n \rightarrow \infty} \frac{ \E[L(X^*_1, X^*_2, \ldots, X^*_n)] }{n^{1/2}}
\leq   2^{-1/2} \beta  < \beta \leq
\limsup_{n \rightarrow \infty} \frac{ \E[L(X^*_1, X^*_2, \ldots, X^*_n)] }{ n^{1/2} };
\end{equation}
where $\beta$ is the BHH constant.
\end{theorem}

Our construction proves more broadly that there are no ergodic analogs for many of the other subadditive Euclidean functionals
such as the Steiner Tree Problem
\citeaffixed{HwangSteinerBOOK,Yuk:SPRI1998}{cf.}. We will return to these and other
general considerations in Section \ref{se:conclusions} where we also describe some open problems, but first we
explain more fully the motivation for Theorems \ref{th:noBHHThm} and \ref{th:noBHHThmExpectations}.

\subsection*{BHH Theorems for Dependent Sequences: Motivation and Evolution}

The traveling salesman problem (or, TSP) has a remarkably extensive literature;
the monographs of \citeasnoun{Lawler1985}, \citeasnoun{Gutin2002} and \citeasnoun{AppBixChvaCook:PRINCETON2006}
note that it is among the of the most studied of all problems in combinatorial optimization.
Moreover, the TSP has had a role in many practical and computational developments.
For example, the TSP provided important motivation for the theory of NP-completeness,
the design of polynomial time approximations, the methods of constraint generation in the theory of linear programming,
and --- most relevant here ---
the design of probabilistic algorithms.

Probabilistic algorithms are of two fundamentally different kinds.
One class of algorithms uses \emph{internal randomization}; for example, one
may make a preliminary randomization of a list before running {\sc QuickSort}.
The other class of algorithms takes the view that the \emph{problem
input follows a probabilistic model}.
Algorithms of this second kind are especially common in application areas such as vehicle routing and the layout of
very large-scale integrated circuits, or VLSI.

The partitioning algorithm proposed by \citeasnoun{Karp1977}
is an algorithm of the second kind that was directly motivated by the BHH theorem.
The partitioning algorithm was further analyzed in \citeasnoun{SteeleMOR1981}
and \citeasnoun{KarpSteele1985}, and now it is well understood that
for any $\epsilon >0$
the partitioning algorithm produces a path in time $O(n \log n)$ that has length that
is asymptotically almost surely within a factor of $1+\epsilon$ of the length of the optimal path.

\citeasnoun{Aro:JACM1998} subsequently discovered a polynomial time algorithm that will determine a
$(1+\epsilon)$-optimal path for \emph{any set} of $n$ points in Euclidean space, and,
while the \citename{Aro:JACM1998} algorithm is of great theoretical interest, the degree of the polynomial time bound depends on
$1/\epsilon$ in a way that limits its practicality. On the other hand, \citename{Karp1977}'s partitioning
algorithm is immanently practical, and it has been widely used, especially in vehicle routing
problems \citeaffixed{BertsimasEtAlOR1990,BertsimasvanRyzin1993,PavBisFazIsl:IEEE2007}{see e.g.}

Still, at the core of \citename{Karp1977}'s partitioning algorithm is the assumption that the problem instance
$X_1, X_2, \ldots, X_n$ can be viewed as a sequence of independent and identically distributed points in $\R^2$.
While this assumption
is feasible in some circumstances, there are certainly many more real-world problems where one would want to accommodate
dependent sequences. In addition to the direct benefits, any time one
establishes an analog of the BHH theorem for dependent sequences,
one simultaneously robustifies a substantial body of algorithmic work. \citationmode{full}

Some extensions of the BHH theorem are easily done. For example,
one can quickly confirm that the analog of the BHH theorem is valid for an infinite sequence
$Y_1, Y_2,  \ldots$ of exchangeable
random variables with compact support. To see why this is so,
one can appeal to the full theorem
of \citeasnoun{BeaHalHam:PCPS1959} which asserts that for a sequence
$X_1, X_2,  \ldots $ of independent, identically distributed random variables with values in
a compact subset of $\R^d$ one has the almost sure limit \citationmode{default}
\begin{equation}\label{eq:BHHd}
\lim_{n \rightarrow \infty} n^{-(d-1)/d }L(X_1,X_2, \ldots, X_n) =
c_d \int_{\R^d} {f(x)}^{(d-1)/d} \ dx,
\end{equation}
where $f$ is the density of the absolutely continuous part of the distribution of $X_1$ and
$0 < c_d < \infty$ is a constant that depends only on the dimension $d \geq 2$.
Now, given an infinite sequences $Y_1, Y_2, \ldots$  of exchangeable
random variables with values with compact support in $\R^d$,
the de Finetti representation theorem
and the limit \eqref{eq:BHHd} tell us that there is a random variable $0 \leq C(\omega) < \infty$ such that
\begin{equation*}\label{eq:deFinetti}
\lim_{n \rightarrow \infty} n^{-(d-1)/d }L(Y_1,Y_2, ..., Y_n) = C \quad \text{almost surely}.
\end{equation*}
This particular extension of the BHH theorem is just a simple corollary of the original BHH theorem.
Later, in Proposition \ref{pr:BHH4LocalUnif}
of Section \ref{sec:LocalUniformity},
we give another extension of the BHH theorem to what we call \emph{locally uniform} processes. This extension
requires some delicate preparation.

There are also more wide-ranging possibilities.
If one views the BHH theorem as a reflection of the evenness of the asymptotic placement
of the observations, then there is a much different way to seek to generalize the BHH theorem.
One can even obtain analogs of the BHH theorem for certain non-random sequences. These results are rather crude,
but in some contexts, such as VLSI planning, these analogs may be more relevant than the traditional BHH theorem.

Here one considers the sequence of points $x_1,x_2, \ldots, x_n $ in $[0,1]^2$
for which one has strong control of
the rectangle discrepancy that is defined by
$$
D_n=D(x_1,x_2, \ldots, x_n)=\sup_{ Q \in \QQ} \left| \frac{1}{n} \sum_{t=1}^n \1_Q(x_t) - \lambda(Q)\right|,
$$
where $\QQ$ is the set of all axis-aligned rectangles $Q \subseteq [0,1]^2$,
$\1_Q(x)$ is the indicator function that equals one if $x \in Q$,
and  $\lambda(Q)$ is the Lebesgue measure of $Q$.
Inequalities of \citeasnoun{Ste:PAMS1980}, more recently
refined by \citeasnoun{Stei:ORL2010},  then suffice to show that if one has
\begin{equation}\label{eq:Dcondition}
nD_n=o(n^\delta) \quad \text{as } n \rightarrow \infty \quad \text{for all } \delta \in (0,1),
\end{equation}
then one also has the pointwise limit
\begin{equation}\label{logTSP}
\lim_{n \rightarrow \infty} \frac{ \log L(x_1,x_2, \ldots, x_n)}{\log{n}} = \frac{1}{2}.
\end{equation}

The leading example of such a sequence is $x_n= (n \phi_1 \!\!\!\mod 1, \, n \phi_2 \!\!\!\mod 1)$ where
$\phi_1$ and $\phi_2$ are  algebraic irrationals that are linearly independent over the rational numbers.
A deep theorem of \citeasnoun{Schmi:TAMS1964}
tells us that the discrepancy of this sequence satisfies the remarkable estimate
$nD_n =O((\log n)^{3+\gamma})$ for all $\gamma >0$, and this is more than one needs for the discrepancy criterion
\eqref{eq:Dcondition}.

In contrast, for an independent uniformly distributed sequence on $[0,1]^2$, one only has $nD_n =O_p(n^{1/2})$, and this
is vastly weaker than the discrepancy condition \eqref{eq:Dcondition}.
Moreover, even when one does have \eqref{eq:Dcondition}, the conclusion \eqref{logTSP}
falls a long distance from what the BHH theorem gives us.

The point of this review is that there are
good reasons to want a theorem of BHH type for dependent sequences.
Some spotty progress has been made; nevertheless,
we are still far away from being able to give a simple, easily checked, criterion for a dependent
compactly supported sequence to satisfy a limit that parallels the BHH theorem.

In the theory of the law of large numbers there are two natural bookends. On one side one has the Kolmogorov law of large numbers
for independent identically distributed random variables with a finite mean, and on the other side one has
the Birkhoff ergodic theorem.
The BHH theorem initiates a strong law theory for the traveling salesman functional, and, in the fullness of
time, there will surely be analogs of that theorem for various classes of dependent random variables. The perfect
bookend for the theory would be a theorem that asserts that stationarity, uniformity, and ergodicity suffice.
Theorems \ref{th:noBHHThm} and \ref{th:noBHHThmExpectations} show that the limit theory of the TSP cannot be bookended so nicely.

\subsection*{Observations on the Proofs and Methods}

The main idea  of the proof of Theorem \ref{th:noBHHThmExpectations} is that one can construct a  stationary process
such that, along a subsequence of $\tau_1< \tau_2< \cdots$ of successive (and
ever larger) times, the ensemble of observations up to
time $\tau_j$ will alternately either look very much like an independent uniformly
distributed sample or else look like a sample that has far too many ``twin cities,'' i.e.~pairs of points that are excessively close
together on a scale that depends on $\tau_j$. To make this idea precise, we use a sequence of parameterized transformation
of stationary processes where each transformation adds a new epoch with too many twin cities --- \emph{at an appropriate scale}.
Finally, we show that one can build a single stationary process with infinitely many such epochs.

This limit process provides the desired example of a stationary, uniform process for which the
minimal length paths have expectations that behave much differently from those of Beardwood, Halton, and
Hammersley. The ergodic process required by Theorem \ref{th:noBHHThm} is then obtained by an
extreme point argument that uses Choquet's representation of a stationary measure with uniform marginals
as a mixture of stationary ergodic measures with uniform marginals.

Finally, there are a two housekeeping observations.
First, the classical BHH theorem \eqref{eq:BHHd} requires the distribution of the observations to have
support in a compact set, so the difference between the length of the shortest \emph{path}
and the length of shortest \emph{tour} through $X_1, X_2, \ldots, X_n$
is bounded by a constant that does not depend on $n$.
Consequently, in limit a theorem such as Theorem \ref{th:noBHHThm}
the distinction between tours and paths is immaterial.
For specificity, all of the analyses done here
are for the shortest path functional.

Also, in Section \ref{se:transformation}, we further note that it is
immaterial here whether one takes the square $[0,1]^2$ with its natural metric, or if one takes the metric on $[0,1]^2$
to be the metric of the flat torus $\T$ that is obtained from the unit square with
the natural identification of the opposing boundary edges.
Theorems \ref{th:noBHHThm} and \ref{th:noBHHThmExpectations} are stated above for the traditional Euclidean metric, but, as we later
make explicit, most of our analysis will be done with respect to the torus metric.

\section{Two Transformations of a Stationary Sequence: \\
the $H_{\epsilon, N}$ and $T_{\epsilon, N}$ Transformations} \label{se:transformation}

Our construction depends on the iteration of procedures that transform a given stationary process into another
stationary process with additional properties.
For integers $a \leq b$, we denote by $[a:b]$ the set
$\{a, a+1, \ldots, b\}$, and,
given a doubly infinite sequence of random variables $\XX= (\ldots, X_{-1}, X_0, X_1, \ldots)$,
we define the $[a: b]$ segment of $\XX$ to be the subsequence
\begin{equation}\label{eq:segment}
\XX{[a:b]}= ( X_a, X_{a+1},\ldots, X_{b} ).
\end{equation}
Now, given an integer $p$, we say that the process $\XX$ is \emph{periodic in distribution with period $p$} if we have
\begin{equation*}\label{eq:DefPeriodic}
\XX{[a:a+k]} \stackrel{d}{=} \XX{[b:b+k]} \quad \text{for all $k \geq 0$ and all $b$ such that } b = a\!\!\!\mod p.
\end{equation*}
This is certainly a weaker condition than  stationarity, but, by an old randomization trick,
one can transform a process that is periodic in distribution to a closely related process that is  stationary.
We will eventually apply this construction infinitely many times,
so to fix ideas and notation, we first recall how it works in the simplest setting.

\begin{lemma}[Passage from Periodicity in Distribution to Stationarity]\label{lm:PassageFromPeriodicity}
If the $\R^d$-valued doubly infinite sequence $\widehat \XX = (\ldots, \widehat X_{-1}, \widehat X_0, \widehat X_1, \ldots)$
is periodic in distribution with period $p$,
and, if $I$ is chosen independently and uniformly from $[0:p-1]$, then the
doubly infinite sequence $\widetilde{\XX} = (\ldots, \widetilde X_{-1}, \widetilde X_0, \widetilde X_1, \ldots)$ defined by setting
$$
\widetilde X_t =\widehat X_{t+I} \quad  \text{for all } t \in \Z
$$
is a stationary process.
\end{lemma}

\begin{proof}
Fix $0\leq j < \infty$ and take Borel subsets $A_0$, $A_1$, \ldots, $A_j$ of  $\R^d$.
By the definition of $\widetilde{\XX}$ and by conditioning on $I$, one then has
\begin{align}
\P(\widetilde X_{t}\in A_0,\, &\widetilde X_{t+1}\in A_1, \ldots, \widetilde X_{t+j}\in A_j) \notag \\
&=
\frac{1}{p}\sum_{i=0}^{p-1} \P(\widehat X_{t+i}\in A_1, \widehat X_{t+1+i}\in A_2, \ldots , \widehat X_{t+j+i}\in A_j)\notag\\
&=
\frac{1}{p}\sum_{i=0}^{p-1} \P(\widehat X_{t+1+i}\in A_1, \widehat X_{t+2+i}\in A_2, \ldots , \widehat X_{t+1+j+i}\in A_j) \label{eq:RandomSumB} \\
&=
\P(\widetilde X_{t+1}\in A_0, \widetilde X_{t+2}\in A_1, \ldots, \widetilde X_{t+1+j}\in A_j), \notag
\end{align}
where the periodicity in distribution of $\widehat \XX$ is used to obtain \eqref{eq:RandomSumB}. Specifically, by periodicity in
distribution, the last summand of \eqref{eq:RandomSumB} is equal to the first summand of the preceding sum.
This tells us that $\P(\widetilde X_{t}\in A_0, \widetilde X_{t+1}\in A_1, \ldots, \widetilde X_{t+j}\in A_j)$ does not depend on
$t$, and, since $j$ is arbitrary, we see that $\widetilde \XX $ is stationary.
\end{proof}

In the analysis of shortest paths in $[0,1]^2$,
there are three useful distances that one can consider.
One has (a) the traditional Euclidean distance,
(b) the torus distance where one identifies opposite boundary edges and \
(c) the ``Free-on-Boundary'' distance where the cost of travel along any boundary edge is taken to be zero.
For the moment, we let the  length of the shortest path through
the points $x_1,x_2, \ldots, x_n$ under the Euclidean, torus, and ``Free-on-Boundary'' distances
be denoted by $L_E(\cdot)$, $L_T(\cdot)$ and $L_B(\cdot)$ respectively, and we
note that one can show that
\begin{align}\label{eq:threeDistances}
L_B(x_1,x_2, \ldots, x_n) & \leq L_T(x_1,x_2, \ldots, x_n) \\
                          & \leq L_E(x_1,x_2, \ldots, x_n) \leq  L_B(x_1,x_2, \ldots, x_n)+4. \nonumber
\end{align}

The first two inequalities are obvious, and the third is easy if one draws the right picture.
Specifically,
given a path of minimal length under the ``Free-on-Boundary" distance, it may or may not connect with the boundary,
but if it does connect with the boundary then
one can always choose a path of minimal ``Free-on-Boundary"
length that never traverses any part of the boundary more than once.
Consequently, if instead of riding for free one were to pay the full Euclidean cost of this travel,
then the added cost of this boundary travel would be generously bounded by the
total length of the boundary, and consequently one has the last inequality of \eqref{eq:threeDistances}.
This argument of \citename{RedYuk:AAP1994} \citeyear{RedYuk:AAP1994,RedYuk:SPA1996}
has natural analogs in higher dimensions and for other functionals, see \citeasnoun[pp.~12--17]{Yuk:SPRI1998} for
further discussion and details.
Incidently, one should also note that the torus model was considered earlier in the analyses of the minimum
spanning tree problem
by \citeasnoun{AvrBer:AOAP1992} and \citeasnoun{Jai:AAP1993}.

An immediate implication of the bounds \eqref{eq:threeDistances}
is that if the asymptotic relation \eqref{eq:BHHlimit}
holds for any one of the three choices for the distance on $[0,1]^2$, then it holds for all three.
Here we will find it convenient to work with the
torus distance, since this choice gives us a convenient additive group structure.
In particular, for $X=(\xi,\xi') \in \T$ and $0 < \epsilon < 1$,
we can define the $\epsilon$-translation $X(\epsilon)$ of $X$, by setting
\begin{equation}\label{eq:translation-def}
X(\epsilon)=(\, (\xi+\epsilon) \!\!\! \mod 1, \, \xi').
\end{equation}
That is, we get $X(\epsilon)$ by shifting $X$ by $\epsilon$ in \emph{just the first coordinate} and the shift is taken modulo $1$.

We now consider a doubly infinite stationary process $\XX=(\ldots , X_{-1}, X_0, X_1,  \ldots)$ where each coordinate
$X_t$ takes its value in the flat torus $\T$.
Given $N \in \N$, we then define blocks  $B_k$, $k \in \Z$, of length $2N$ by setting
\begin{equation}\label{eq:block}
B_k=( X_{kN}, X_{kN+1}, \ldots, X_{(k+1)N-1}, X_{kN}(\epsilon), X_{kN+1}(\epsilon), \ldots, X_{(k+1)N-1}(\epsilon) ),
\end{equation}
where the translations $X_t(\epsilon)$ for $t \in [kN:(k+1)N-1]$, are defined as in \eqref{eq:translation-def}.
We write the doubly infinite concatenation of these blocks as
$$
(\ldots, B_{-2}, B_{-1}, B_0, B_1, B_2, \ldots),
$$
and we note that this gives us a doubly infinite sequence of $\T$-valued  random variables that we may also write as
\begin{equation*}
\widehat \XX=(\ldots, \widehat X_{-2}, \widehat X_{-1}, \widehat X_0, \widehat X_1, \widehat X_2, \ldots ).
\end{equation*}
The process $\widehat \XX=( \widehat X_t : t \in \Z )$ is called the \emph{hat process},
the passage from $\XX$ to $\widehat \XX$ is called a $H_{\epsilon, N}$ \emph{transformation}, and we write
\begin{equation}\label{eq:H-transformation}
 \widehat{\XX} = H_{\epsilon, N} ( \XX ).
\end{equation}

It is useful to note that the hat process $\widehat \XX=(\widehat X_t : t \in \Z)$ is periodic in distribution with period $2N$, so
one can use Lemma \ref{lm:PassageFromPeriodicity} to construct a closely related stationary sequence
$\widetilde \XX =(\widetilde X_t  : t \in \Z)$. Specifically, we set
\begin{equation*}
\widetilde X_t =\widehat X_{t+I}\quad \text{for all } t \in \Z,
\end{equation*}
where the random index $I$ has the uniform distribution on $[0:2N-1]$ and $I$ is independent of the sequence $\widehat \XX$.
The complete passage from $\XX$ to $\widetilde \XX$ is called a $T_{\epsilon, N}$ \emph{transformation}, and it is denoted by
\begin{equation*}\label{eq:T-transformation}
\widetilde \XX = T_{\epsilon, N} (\XX).
\end{equation*}
We will make repeated use of the two-step nature of this construction, and we stress that
the hat process $\widehat \XX$ is more than an intermediate product. The properties of
the hat process $\widehat \XX$ are the real guide to our constructions, and the
stationary process $\widetilde \XX$ is best viewed as a polished version of $\widehat \XX$.

\subsection*{Properties of the $T_{\epsilon, N}$ transformation}\label{sec:PreservedProperties}

The process $\widetilde \XX$ that one obtains from $\XX$  by a $T_{\epsilon, N}$ transformation
retains much of the structure of $\XX$. We begin with a simple example.

\begin{lemma}[Preservation of Uniform Marginals]\label{lm:UniformPreservation}
If $\XX = (X_t: t \in \Z)$ is a doubly infinite $\T$-valued process such that
$X_t$ has the uniform distribution on $\T$ for each $t \in \Z$ and if
$$
\widetilde \XX = T_{\epsilon, N} (\XX) = (\widetilde X_t: t \in \Z) \quad \text{where } 0< \epsilon < 1 \text{ and } N \in \N,
$$
then $\widetilde X_t$ has the uniform distribution on $\T$ for each $t \in \Z$.
\end{lemma}

\begin{proof}
Here it is immediate from the blocking and shifting steps of the $H_{\epsilon,N}$ transformation that $\widehat{X}_t$ has the uniform
distribution on $\T$ for each  $t \in \Z$.  Also, by construction, the distribution of $\widetilde X_t$ is then a mixture of
uniform distributions, and therefore $\widetilde X_t$ also has the uniform distribution on $\T$.
\end{proof}

Given a doubly infinite  sequence $\XX = ( \ldots, X_{-1}, X_0, X_1, \ldots )$  the \emph{translation} of $\XX$
is the process $\XX(\delta)$ defined by setting
$$\XX(\delta) = (X_t (\delta): t \in \Z),$$
where,
as before, we have $X_t= (\xi_t, \, \xi_t') \in \T$ and  $X_t (\delta) = (\xi_t + \delta , \, \xi_t')$ with
the addition in the first coordinate that is taken modulo one. We also
say that $\XX$ is \emph{translation invariant} if for each $\delta > 0 $ we have $\XX \stackrel{d}{=}\XX(\delta)$.
Next, we check that translation invariance of a process $\XX$
is preserved under any $T_{\epsilon, N}$ transformation.

\begin{lemma}[Preservation of Translation Invariance]\label{lm:translation-invariance}
If $\XX$ is a doubly infinite, translation invariant, $\T$-valued process, then for each
$0< \epsilon < 1$ and $N \in \N$, the process $\widetilde \XX$ defined by
$$
\widetilde \XX = T_{\epsilon, N} (\XX)
$$
is also translation invariant.
\end{lemma}

\begin{proof}
The crux of the matter is that a translation by $\delta$ and an application of the $H_{\epsilon,N}$ transformation
are pointwise commutative. In symbols one has
\begin{equation}\label{eq:H-translation-equality}
H_{\epsilon,N}(\XX(\delta)) = ( H_{\epsilon,N}(\XX)) (\delta),
\end{equation}
where on the left we translate $\XX$ by $\delta$
and then apply $H_{\epsilon,N}$, and
on the right we apply $H_{\epsilon,N}$ and then translate by $\delta$.
A formal proof of this identity only requires one to unwrap the definition of $H_{\epsilon,N}$
and to use commutativity of addition.

Now, by our hypothesis, $\XX$ is equal in distribution to  $\XX(\delta)$, and, since
equality in distribution is preserved by a $H_{\epsilon,N}$ transformation, we also have
\begin{equation}\label{eq:H-translation-invariance}
H_{\epsilon,N}(\XX) \stackrel{d}{=}  H_{\epsilon,N}(\XX(\delta)) =( H_{\epsilon,N}(\XX)) (\delta),
\end{equation}
where in the second equality we used \eqref{eq:H-translation-equality}.
When we shift the indices of two processes that are equal in distribution by an independent shift $I$, the resulting processes
are again equal in distribution. Thus, when make the shift on each side of \eqref{eq:H-translation-invariance}
that is required by the definition of the
$T_{\epsilon,N}$ transformation, we have
$$
\widetilde \XX=T_{\epsilon,N}(\XX) \stackrel{d}{=}( T_{\epsilon,N}(\XX)) (\delta)=\widetilde \XX (\delta),
$$
just as we needed.
\end{proof}

The process $\widetilde \XX$ that one obtains from a doubly infinite stationary sequence $\XX$
by a $T_{\epsilon, N}$ transformation is typically singular with respect to $\XX$.
Nevertheless,  on a short segment the two processes are close in distribution. The next lemma
makes this precise.

\begin{lemma}[Closeness in Distribution]\label{lm:CloseInDistribution}
Let $\XX$ be a translation invariant doubly infinite stationary sequence with values in the flat torus $\T$.
For each $0 < \epsilon < 1$ and $N \in \N$, the process $\widetilde \XX$ defined by
$
\widetilde \XX = T_{\epsilon, N} (\XX)
$
satisfies
\begin{equation}\label{eq:ClosenessIneq}
|\, \P(\widetilde \XX[0:m] \in \A) - \P(\XX[0:m] \in \A) \,| \leq \frac{m}{N},
\end{equation}
for all Borel sets $\A \subseteq \T^{m+1}$ and for all  $m=0, 1,2, \ldots$.
\end{lemma}

\begin{proof}
Recalling the two-step construction that takes one from  $\XX$ to $\widetilde \XX$,
we first note that we can write $\widetilde \XX[0:m]$ in terms of
the hat process $\widehat \XX$ given by the construction \eqref{eq:H-transformation}; specifically, we have
\begin{equation*}
\widetilde \XX[0:m]=\widehat \XX[I:I+m],
\end{equation*}
where the random variable $I$ is independent of $\widehat \XX$ and uniformly distributed on $\{0, 1, \ldots,  2N-1\}$.
Now we condition on the value $i$ of $I$.
For any $i$ such that $[i:i+m] \subseteq [0:N-1]$, the definition of the hat process gives us the distributional identity
\begin{equation}\label{eq:eqdist-1}
\widehat \XX[i:i+m]=\XX[i:i+m] \stackrel{d}{=} \XX[0:m],
\end{equation}
where in the last step we used the stationarity of $\XX$. Similarly, for $i$ such that $ [i:i+m] \subseteq [N: 2N-1]$, we have
\begin{equation}\label{eq:eqdist-2}
\widehat \XX[i:i+m]=\XX[i-N:i-N+m](\epsilon) \stackrel{d}{=} \XX[i-N:i-N+m]\stackrel{d}{=} \XX[0:m],
\end{equation}
where, in the next-to-last step, we use the translation invariance of $\XX$
and in the last step we again used the stationarity of $\XX$.

We now consider the ``good set'' of indices
$$
G=\{i: 0 \leq i \leq i+m < N \text{ or } N \leq i \leq i+m < 2N\},
$$
where the equalities \eqref{eq:eqdist-1} and \eqref{eq:eqdist-2} hold,
and we also consider the complementary ``bad set'' of indices
$
 G' = [0:2N-1] \backslash G.
$
If we condition on $I$ and use  \eqref{eq:eqdist-1} and \eqref{eq:eqdist-2}, then we see that for our Borel set
$\A \subseteq \T^{m+1}$ one has
\begin{align*}
\P(\widetilde \XX[0:m] \in \A)
& = \frac{1}{2 N} \sum_{i \in G} \P ( \widehat \XX[i:i+m] \in \A )
  + \frac{1}{2 N} \sum_{i \in G'} \P ( \widehat \XX[i:i+m] \in \A ) \\
& = \frac{2 N - 2 m }{2 N} \, \P (\XX[0:m] \in \A )
  + \frac{1}{2 N} \sum_{i \in G'} \P ( \widehat \XX[i:i+m] \in \A ),
\end{align*}
which one can then write more nicely as
\begin{align*}
\P  (\widetilde \XX[0:m] \in \A )  - \P (\XX[0:m] \in \A )
& =   - \frac{ m }{ N} \, \P (\XX[0:m] \in \A ) \\
& + \frac{1}{2 N} \sum_{i \in G'} \P ( \widehat \XX[i:i+m] \in \A ).
\end{align*}
The last sum has only $|G'| = 2 m $ terms, so we have the bounds
$$
- \frac{m }{N} \leq \P (\widetilde \XX[0:m] \in \A )  - \P (\XX[0:m] \in \A )  \leq \frac{m}{N}
$$
that complete the proof of the lemma.
\end{proof}

\section{Locally Uniform Processes and BHH in Mean}\label{sec:LocalUniformity}

Our inductive construction requires an extension of
the Beardwood, Halton and Hammersley theorem to a certain class of dependent processes that we
call \emph{locally uniform processes}. The definition of these processes requires some notation.

First, given any Borel set $A \subseteq \T$ and any set of indices $J \subseteq \Z$, we let
$$
N(A, \YY[J]) = \sum_{t \in J} \1(Y_t \in A),
$$
so $N(A, \YY[J])$ is the number of elements of $\YY[J]=(Y_t: t \in J)$ that fall in the Borel set $A$.
We also say that a subset $Q \subseteq \T$ is a \emph{subsquare of side length $\alpha$} if it
can be written as $[x, x+\alpha] \times [y, y+\alpha]$
where one makes the usual
identifications of the points in the flat torus.
Finally, given any $0 < \alpha \leq 1$, we let $\QQ(\alpha)$ denote the set of all subsquares of $\T$ that have
side length less than or equal to $\alpha$.

We further let
$$
(U_s(A): 1 \leq s < \infty)
$$
denote an infinite sequence of independent random variables with
the uniform distribution on the Borel set $A \subseteq \T$; in particular, $\S = \{U_s(A): 1 \leq s \leq n\}$ is a
uniform random sample from $A$ with cardinality $|\S|=n$.
Finally, if  $\S$ and $\S'$  are two random finite subsets of  the Borel set $A \subseteq \T$, we write
$$
\S \stackrel{\rm dpp}{=} \S'
$$
to indicate that $\S$ and $\S'$  are equal in distribution as point processes.

\begin{definition}[Locally Uniform Processes]\label{def:locally-uniform-process}
If $0 < \alpha \leq 1$ and $M \in \N$ we say that a $\T$-valued process $\YY=(Y_t: t \in \Z )$
with uniform marginal distributions
is an \emph{$(\alpha, M)$ locally uniform process}
provided that it satisfies the two following conditions:

\begin{enumerate}[align=left,labelindent=!,labelwidth=!,labelsep=0pt]
    \item[\eqnum\label{eq:VarCondition}] \textsc{Variance Condition.}
            There is a constant $C < \infty$ that depends only on the distribution of $\YY$
            such that for each pair of integers $a\leq b$ and each Borel set $A \subseteq Q \in \QQ(\alpha)$ we have
            \begin{equation*}
            \Var [ N(A, \YY[a:b]) ] \leq  C|b-a+1|.
            \end{equation*}

    \item[\eqnum\label{eq:ConditionallyUniform}] {\sc Local Uniformity Condition.}
            For each pair of integers $a\leq b$ and each Borel set $A \subseteq Q \in \QQ(\alpha)$
            there is a random set
            $$
            \S \subseteq \{ Y_t:  Y_t \in A \text{ and } t \in [a:b]\}
            $$
            for which one has the cardinality bounds
            $$
            N(A, \YY[a:b]) - M \leq | \S | \leq N(A, \YY[a:b]),
            $$
            and the distributional identity
            \begin{equation*}
            \S \stackrel{\rm dpp}{=} \{U_s(A): 1 \leq s \leq |\S| \}.
            \end{equation*}
\end{enumerate}
\end{definition}

Here one should note that if
$\alpha = 1$ and $M = 0$, then a locally uniform process is just a sequence of independent
random variables with the uniform distribution on the flat torus $\T$.
More generally,
the parameter $\alpha$ quantifies the scale at which
an $(\alpha, M)$ locally uniform process looks almost like a sample of independent uniformly distributed random points,
and the parameter $M$ bounds the size of an exception set that can be discarded to
achieve exact uniformity.

One should also note that in a locally uniform process $\YY =(Y_t: t \in \Z )$ each $Y_t$ has the uniform distribution on $\T$, but
$\YY$ is not required to be a stationary process. This will be important to us later.
The first observation is that
despite the possible lack of stationarity, one can still show that
locally uniform processes satisfy a relaxed version of the BHH theorem.

\begin{proposition}[BHH in Mean for Locally Uniform Processes]\label{pr:BHH4LocalUnif}
If the $\T$-valued process $\YY =(Y_t: t \in \Z )$ is
$(\alpha,M)$ locally uniform for some $0 < \alpha \leq 1$ and some $M \in \N$, then one has
\begin{equation*}\label{eq:BHH4alpha}
\E[L(Y_1, Y_2, \ldots, Y_n)] \sim \beta n^{1/2} \quad \text{ as } n \rightarrow \infty,
\end{equation*}
where $\beta$ is the BHH constant in \eqref{eq:BHHlimit}.
\end{proposition}

\begin{proof}
For any integer $k$ such that $k^{-1} \leq \alpha$ we consider the natural
decomposition of the flat torus into $k^2$ subsquares $Q_i$, $i=1,2, \ldots, k^2$,
of side length $k^{-1}$. We then introduce the sets
$$
S(Q_i, n) =\{Y_t: Y_t \in Q_i \text{ and } t \in [1 : n] \}
$$
and we let $L(S(Q_i,n))$ be the length of the shortest path through the points in $S(Q_i, n)$.
If we then stitch these $k^2$ optimal paths
together by considering the subsquares $Q_i$, $1 \leq i \leq k^2$,
in plowman's order --- down one row then back the next, then our
stitching cost is less than $3k$, but all we need from these considerations is that
there is a universal constant $C_1>0$ such that
one has the pointwise bound,
\begin{equation}\label{eq:pointwiseUpper}
L(Y_1, Y_2, \ldots, Y_n)
\leq
C_1 k + \sum_{i=1}^{k^2} L(S(Q_i,n)).
\end{equation}
More notably, one can also show that there is a universal constant $C_0>0$ for which one has
\begin{equation}\label{eq:pointwiseLower}
-C_0 k + \sum_{i=1}^{k^2} L(S(Q_i,n)) \leq L(Y_1, Y_2, \ldots, Y_n).
\end{equation}
This bound is due to \citeasnoun{RedYuk:AAP1994}, and it may be proved by noticing that
the sum of the values $L(S(Q_i,n))$ can be bounded
by the length of the optimum path through $Y_1, Y_2, \ldots, Y_n$
and the sum of lengths of the boundaries of the individual squares $Q_i$, $1 \leq i \leq k^2$.
For the details concerning \eqref{eq:pointwiseLower}, including analogous bounds for $[0,1]^d$, $d \geq 2$, one can consult
\citeasnoun[Chapter 3]{Yuk:SPRI1998}.

We now recall that
$$
S(Q_i, n) =\{Y_t: Y_t \in Q_i, \, t \in [1 : n] \} \quad \text{ and } \quad
N(Q_i, \YY[1:n]) = \sum_{t=1}^n \1(Y_t \in Q_i),
$$
and we estimate the value of $\E [L(S(Q_i,n))]$ as $n\rightarrow \infty$.
By the $(\alpha,M)$ local uniformity of $\YY$ and \eqref{eq:ConditionallyUniform},
there is a set $\S_{n} \subseteq S(Q_i,n)$ such that
\begin{equation}\label{eq:cardinality-bound-Sn}
N(Q_i, \YY[1:n]) - M  \leq |\S_n| \leq N(Q_i, \YY[1:n])
\end{equation}
and
$$
\S_{n} \stackrel{\rm dpp}{=} \{U_s(Q_i): 1 \leq s \leq |\S_n| \}.
$$
From the cardinality bounds \eqref{eq:cardinality-bound-Sn} and a crude path length comparison we then have
\begin{equation}\label{eq:lenghtinQ}
| L(\S_n) - L(S(Q_i,n))| \leq  2^{3/2} M k^{-1}.
\end{equation}
Now, given the pointwise bounds \eqref{eq:pointwiseUpper}, \eqref{eq:pointwiseLower} and \eqref{eq:lenghtinQ},
the lemma will follow once we show that
\begin{equation}\label{eq:BabyBHH}
\E[ L(\S_n) ]\sim \beta k^{-2} n^{1/2} \quad\quad \text{as } n \rightarrow \infty.
\end{equation}
If we now let $\ell(j)$ denote the expected length of the minimal path through an independent uniform sample of size $j$
in the unit square, then by scaling and by conditioning on the cardinality of $\S_n$
we have
\begin{align*}
\E[L(\S_n)]=k^{-1} \sum_{j=0}^n \ell(j) \P(|\S_n|=j).
\end{align*}
The BHH relation \eqref{eq:BHHlimit} also tells us that we have
$
\ell(j) =  \beta j^{1/2} + o( j^{1/2}),
$
so
\begin{align} \label{eq:EV-TSP-in-squares}
\E[L(\S_n)]
& = \beta k^{-1} \sum_{j=0}^n j^{1/2} \,\, \P(|\S_n|=j) + o (n^{1/2}) \\
& = \beta k^{-1} \E[|\S_n|^{1/2}] + o (n^{1/2}).\nonumber
\end{align}

Now one just needs to estimate $\E[|\S_n|^{1/2}]$.
Linearity of expectation and the fact that each $Y_t$ is uniformly distributed on $\T$ then combine to give us
\begin{equation}\label{eq:Nupper}
\E[ N(Q_i, \YY[1:n]) ] = k^{-2} n \quad \text{and} \quad  \E[ N(Q_i, \YY[1:n])^{1/2} ] \leq  k^{-1} n^{1/2},
\end{equation}
where the second bound comes from Jensen's inequality.
For any $0<\theta <1$, the Variance Condition \eqref{eq:VarCondition}
and Chebyshev's inequality also give us
$$
\P(N(Q_i, \YY[1:n])\geq \theta \, \E [N(Q_i, \YY[1:n])])=1-O(1/n),
$$
so we have the lower bound
\begin{equation}\label{eq:Nlower}
 \theta^{1/2} k^{-1} n^{1/2} (1-O(1/n) ) \leq \E [ N(Q_i, \YY[1:n])^{1/2}].
\end{equation}
By the inequalities \eqref{eq:Nupper} and \eqref{eq:Nlower}, we see from the arbitrariness of $\theta$ and the bound \eqref{eq:cardinality-bound-Sn}
that we have
\begin{equation*}\label{ex:root-upper-bound}
\E [ |\S_n|^{1/2} ]\sim k^{-1} n^{1/2} \quad \quad \text{as } n \rightarrow \infty.
\end{equation*}
Finally, from \eqref{eq:EV-TSP-in-squares} we see this gives us that $\E[L(\S_n)]\sim \beta k^{-2} n^{1/2}$ as $n \rightarrow \infty$,
so by the observation preceding \eqref{eq:BabyBHH} the proof of the lemma is complete.
\end{proof}

\section{Preservation of Local Uniformity}\label{se:LocalUnif-H-transformation}

A key feature of local uniformity is that it is preserved by a $H_{\epsilon, N}$ transformation.
Before proving this fact, it is useful to introduce some notation and
to make a geometric observation.  First, for
any Borel set $A \subseteq \T$ we consider two kinds of translations of $A$ by  an $\epsilon>0$; these are given by
$$
^\epsilon \! A = \{(x-\epsilon,y): (x,y)  \in A \} \quad \text{and} \quad A^\epsilon  = \{(x+\epsilon,y): (x,y)  \in A \},
$$
so the set $^\epsilon\! A$ is equal to the set $A$ translated to the \emph{left} in $\T$,
and  $A^\epsilon$ is equal to the set $A$
shifted to the \emph{right} in $\T$.

Next, we fix a Borel set $A \subseteq \T$ for which we have
$^\epsilon \!A \cap A =\emptyset.$ We then consider a sequence of independent random variables $Y_1, Y_2, \ldots, Y_n$ with the uniform
distribution on $^\epsilon \!A \cup A$ and a sequence  of independent random variables $Y_1', Y_2', \ldots, Y_n'$ with the uniform
distribution on $A$. In this situation, we then have an elementary distributional identity of point processes,
\begin{equation}\label{eq:ElemID}
\{ Y_t: Y_t \in {^\epsilon \! A},\,  t \in [1:n] \}^\epsilon \cup \{ Y_t: Y_t \in A, \,  t \in [1:n] \}
\stackrel{\rm dpp}{=}
\{ Y_t':  t \in [1:n] \},
\end{equation}
where one should note that in the first term of \eqref{eq:ElemID} we have used both a left shift on $A$ and a right shift on
the set of points that fall into the shifted set $^\epsilon \! A$. The identity \eqref{eq:ElemID} provides an essential
step in the proof of the next proposition.

\begin{proposition}[Local Uniformity and $H_{\epsilon, N}$ Transformations]\label{pr:local-unif}
If $\XX$ is an $(\alpha,M)$ locally uniform process
with $0< \alpha \leq 1$ and $M \in \N$ and if one has $0 < \epsilon < \alpha$ and $N < \infty$, then the process
$$
\widehat \XX = H_{\epsilon, N} (\XX)
$$
is $(\widehat{\alpha}, \widehat{M})$ locally uniform with
$$
0< \widehat{\alpha} < \min\{ \epsilon, \alpha - \epsilon\} \quad \text{ and } \quad
\widehat{M}=M + 4N.
$$
\end{proposition}

\begin{proof}
The definition of $H_{\epsilon, N}$  tells us
that for each $a\leq b$ we can write $\widehat \XX[a:b]$ in the form
\begin{equation}\label{eq:blockDecomp}
\widehat \XX[a:b]=(L,B_{\ell},B_{\ell+1}, \ldots, B_{r}, R).
\end{equation}
where the lengths of the segments $L$ and $R$ are between $0$ and $2N-1$ and where
each segment $B_k$,  $k \in [\ell: r]$,  is a ``complete block'' of $\widehat \XX$  that has the form
\begin{equation}\label{eq:blocki2}
B_k=( X_{k N}, X_{k N+1}, \ldots, X_{(k+1)N-1}, X_{k N}(\epsilon), X_{k N+1}(\epsilon), \ldots, X_{(k+1)N-1}(\epsilon) ).
\end{equation}
Here, in the case that $\widehat \XX[a:b]$ does not contain a
complete block, we simply write $\widehat \XX[a:b]=(L, R)$ for any choices of $L$ and $R$ that satisfy
the length constraints. In this case we (crudely) note that $|b-a+1| < 4N$.

If the set of complete blocks in the decomposition \eqref{eq:blockDecomp} is not empty, then we can take
\begin{equation}\label{eq:blockDecomp2}
\widehat \XX[a':b']=(B_\ell,B_{\ell + 1}, \ldots, B_{r}),
\end{equation}
to be the maximal subsegment of $\widehat \XX[a:b]$ that contains only complete blocks.
In this case, we have
$$
a' = \min\{2 k N: 2 k N \geq a\} \quad \text{and} \quad b'=\max\{2 k N -1 : 2 k N - 1\leq b\},
$$
but we will not make use of these explicit formulas beyond noting that we have the bounds
$0 \leq |b-a+1|-|b'-a'+1| \leq 4N$.

For each $k \in [\ell:r]$ we let $B_k'$ be the first half of the complete block $B_k$
described in \eqref{eq:blocki2}; that is, we set
$$
B_k'=( X_{k N}, X_{k N+1}, \ldots, X_{(k+1)N-1}).
$$
We then concatenate these half blocks to obtain a segment of the original $\XX$ process;
specifically we obtain a segment of $\XX$ that
can be written as
$$
\XX[a'':b'']=(B_\ell', B_{\ell+1}', \ldots, B_r').
$$
Here one should note that the
length of the segment $\XX[a'':b'']$ is exactly half of the length of the segment $\widehat \XX[a':b']$.
We will use the correspondence between $[a:b]$, $[a':b']$, and $[a'':b'']$ throughout the remainder of this proof.


\begin{figure}[t]
    \caption{Local Uniformity and $H_{\epsilon,N}$ Transformations}\label{fig:translations}
    \begin{tikzpicture}

    \draw[thick] (0,0) -- (6,0) -- (6,6) node[anchor=north east] {$\boldsymbol{Q'}$} -- (0,6) -- (0,0);

    \draw  [thick, decoration={brace,mirror,raise=0.125cm},decorate]
     (0,0) -- (6,0) node [pos=0.5,anchor=north,yshift=-0.25cm] {$\alpha$};

    \draw[] (0.25,1.75) -- (2.5,1.75) -- (2.5,4) node[anchor=north east] {${^\epsilon Q}$} -- (0.25,4) -- (0.25,1.75);

    \node[cloud, cloud puffs=4.9, cloud ignores aspect, minimum width=1cm, minimum height=1cm, align=center, draw, red, thick] (cloud) at (1.0, 2.5) {$^\epsilon \! A$};


    \draw  [decoration={brace,mirror,raise=0.125cm},decorate]
     (0.27,1.75) -- (2.48,1.75) node [pos=0.5,anchor=north,yshift=-0.25cm] {$\widehat \alpha$};


    \draw  [decoration={brace,mirror,raise=0.125cm},decorate]
     (0.27,1.0) -- (3.23,1.0) node [pos=.5,anchor=north,yshift=-0.25cm] {$\epsilon$};

    \draw (3.25,1.75) -- (5.5,1.75) -- (5.5,4) node[anchor=north east] {$Q$} -- (3.25,4) -- (3.25,1.75) ;
    \node[cloud, cloud puffs=4.9, cloud ignores aspect, minimum width=1cm, minimum height=1cm, align=center, draw, red, thick] (cloud) at (4.0, 2.5) {$A$};


    \draw  [decoration={brace,mirror,raise=0.125cm},decorate]
     (3.27,1.75) -- (5.48,1.75) node [pos=0.5,anchor=north,yshift=-0.25cm] {$\widehat \alpha$};

    \end{tikzpicture}
    \footnotetext{The figure illustrates the relations \eqref{eq:TwoConditionsOnA}.
    The size constraint on $\widehat \alpha$  guarantees the existence of a subsquare $Q'$ of side length $\alpha$
    that contains both a given $Q \in \QQ(\widehat \alpha)$ and its left translate $^\epsilon Q$. Similarly, the size constraint
    on $\epsilon$ guarantees that $Q$ and $^\epsilon Q$ are disjoint, so the Borel sets
    $A \subseteq Q$ and $^\epsilon \! A \subseteq  {^\epsilon Q}$
    are also disjoint.}
\end{figure}
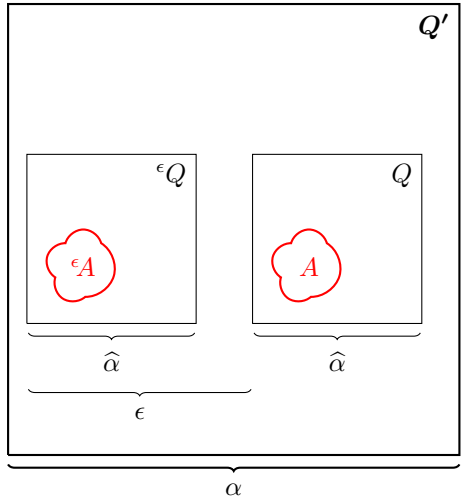


To confirm the
local uniformity of $\widehat \XX$, we need to check the two conditions \eqref{eq:VarCondition} and \eqref{eq:ConditionallyUniform}.
To check \eqref{eq:VarCondition}, we first
fix a Borel set $A \subseteq Q \in \QQ(\widehat{\alpha})$, and we note that
our assumptions $0< \widehat{\alpha} < \min\{ \epsilon, \alpha - \epsilon\}$ and
$0 < \epsilon < \alpha$ give us the relations
\begin{equation}\label{eq:TwoConditionsOnA}
^\epsilon \!A \cap A = \emptyset \quad \text{and} \quad ^\epsilon \! A \cup A \subseteq Q' \quad \text{for a } Q' \in \QQ(\alpha),
\end{equation}
which are illustrated in Figure \ref{fig:translations} on page \pageref{fig:translations}.

Next, we consider an integer pair $a<b$, and we assume for the moment that $\widehat \XX[a:b]$ contains at least one complete block.
By our decomposition \eqref{eq:blockDecomp} of $\widehat \XX[a:b]$ we have
$0\leq |b - a + 1 |-|b' - a' + 1| \leq 4N$, so there is a random variable $W(A, \widehat \XX[a:b])$
such that $0 \leq W(A,\widehat \XX[a:b]) \leq 4N$ and
\begin{align}\label{eq:counts-fromX-toXhat}
N(A, \widehat \XX[a,b]) &= N(A, \widehat \XX[a':b']) + W(A,\widehat \XX[a:b]) \\
&=N (^\epsilon \!A \cup A , \XX[a'':b'']) + W(A,\widehat \XX[a:b]), \nonumber
\end{align}
where in the second equality we just used the definition of $\widehat \XX$ and the definition of the set of indices $[a'':b'']$.

When we compute variances in \eqref{eq:counts-fromX-toXhat}, we get
\begin{align}\label{eq:VarConConf}
\Var[ N(A, \widehat \XX[a:b]) ] &\leq 2 \Var [ N (^\epsilon \!A \cup A , \XX[a'':b'']) ]  + 2 (4N)^2  \\
& \leq 2 C |b''-a''+1| +  32N^2 \leq C |b-a+1| + 32N^2, \notag
\end{align}
where in the second inequality we used the variance condition \eqref{eq:VarCondition}  for $\XX$
and the bounds $2|b'' - a'' +1 | = |b' - a' + 1| \leq |b - a + 1|$.
By \eqref{eq:VarConConf} we see that $\widehat \XX$ satisfies the variance condition \eqref{eq:VarCondition},
and one can take $\widehat C= C + 32N^2$ as a generous choice for the required constant.
Finally, we note that in case $\widehat \XX[a:b]$ does not contain a complete block, then we have $|b-a+1|< 4N$. In this case, it is trivial that
one has the variance condition \eqref{eq:VarCondition} for $\widehat \XX[a:b]$ when one takes $\widehat C$ for the
required constant.

To check the local uniformity condition \eqref{eq:ConditionallyUniform} for $\widehat \XX$, we
fix $a \leq b$ and $A$ as before so we continue to have the relations \eqref{eq:TwoConditionsOnA}.
In the case when $\widehat \XX[a:b]$ does not contain a complete block, then $|b-a+1|< 4N$ and we can simply take $\S$
required by \eqref{eq:ConditionallyUniform}
to be the empty set. In this case the conditions required by  \eqref{eq:ConditionallyUniform}
all hold trivially. Thus,
without loss of generality, we assume that
$\widehat \XX[a:b]$ contains at least one complete block.

By our hypothesis,  we have condition \eqref{eq:ConditionallyUniform} for the process $\XX$,
and in this case we can apply \eqref{eq:ConditionallyUniform} to the Borel set $^\epsilon \! A \cup A   \subseteq Q \in \QQ(\alpha)$
and the interval $[a'':b'']$. We are then guaranteed the existence of a set
\begin{equation}\label{eq:S-inclusion}
\S \subseteq \{ X_t:  X_t \in {^\epsilon\!A} \cup A \text{ and } t \in [a'':b'']  \}
\end{equation}
for which we have both the cardinality bounds
\begin{equation}\label{eq:cardinality-S}
 N (^\epsilon \!A \cup A , \XX[a'':b'']) - M \leq | \S | \leq  N (^\epsilon \!A \cup A , \XX[a'':b''])
\end{equation}
and the distributional identity
$
\S \stackrel{\rm dpp}{=} \{U_s(^\epsilon \!A \cup A ): 1 \leq s \leq |\S| \}.
$
Now it only remains to check that if we take
$$
\widehat \S = \{\S \cap \,  ^\epsilon\!\!A \}^\epsilon \cup \{\S \cap A \} \subseteq A
$$
then the set $\widehat \S$ will suffice to confirm that we have the
local uniformity condition \eqref{eq:ConditionallyUniform} for $\widehat \XX$.

First, we note from the  inclusion \eqref{eq:S-inclusion} that we have
\begin{align}\label{eq:Shat-inclusion}
\widehat \S
& \subseteq \{ X_t:  X_t \in {^\epsilon\!A} \text{ and } t \in [a'':b'']  \}^\epsilon \cup \{ X_t:  X_t \in A \text{ and } t \in [a'':b'']  \} \\
& =  \{ \widehat X_t:  \widehat X_t \in A \text{ and } t \in [a':b']  \} \subseteq \{ \widehat X_t:  \widehat X_t \in A \text{ and } t \in [a:b]  \}, \nonumber
\end{align}
where, for the middle equality we used the fact that $\widehat \XX[a':b']$
contains only complete blocks and the fact that  $\XX[a'':b'']$
contains only the unshifted halves of the same blocks.
The last inclusion follows simply because $a \leq a' < b' \leq b$.

Now, we already know from \eqref{eq:ElemID} that the point process $\widehat{\S}$
has the same distribution as an independent sample of $|\widehat{\S}|$
points with the uniform distribution on $A$, so to complete the proof of the lemma, we just need to check the double bound
\begin{equation}\label{eq:doublebound}
N(A, \widehat \XX[a,b]) - M - 4N \leq | \widehat \S | \leq N(A, \widehat \XX[a,b]).
\end{equation}
The upper bound follows from the last inclusion of \eqref{eq:Shat-inclusion}.
For the lower bound, we start with the decomposition \eqref{eq:counts-fromX-toXhat}. The uniform bound $W(A,\widehat \XX[a:b]) \leq 4N$
and the left inequality of \eqref{eq:cardinality-S} then give us
$$
N(A, \widehat \XX[a,b]) - M - 4N \leq |\S|=| \widehat \S |,
$$
and this lower bound completes the check of \eqref{eq:doublebound}.
\end{proof}

Proposition \ref{pr:local-unif} takes us much closer to
the proof of Theorems \ref{th:noBHHThm} and \ref{th:noBHHThmExpectations}. In particular, it is now
easy to show that $T_{\epsilon, N}$ also preserves local uniformity.  We establish this as a corollary to a
slightly more general lemma.

\begin{lemma}[Independent Random Shift of Index]\label{lm:YTransformation}
Let $\YY$ be an $(\alpha, M)$ locally uniform process with $0 < \alpha \leq 1$ and $M \in \N$.
If the random variable $I$ has the
uniform distribution on $\{0,1,\ldots, K-1\}$ and if $I$ is independent of $\YY$, then
the process $\YY'=(Y'_t: t \in \Z)$ defined by
setting
$$
Y_t'=Y_{t+I} \quad \text{for all } t \in \Z
$$
is an $(\alpha, M+K)$ locally uniform process.
\end{lemma}
\begin{proof}
If $ b < a + K $, then the result is trivial since one can simply take $\S=\emptyset$.
On the other hand, if $ a + K \leq b$  then $0 \leq I < K$, and  we have
$$
a \leq a + I < a + K \leq b \leq b+I < b+K.
$$
Thus, for any $A  \subseteq Q \in \QQ(\alpha)$ we have
$$
N(A, \YY[a+K:b]) \leq N(A, \YY'[a:b]) \leq N(A, \YY[a:b+K]),
$$
and we can write
\begin{equation}\label{eq:N-counts-randomized}
N(A, \YY'[a:b]) =  N(A, \YY[a+K:b]) + W( A, \YY'[a:b] ),
\end{equation}
where $W( A, \YY'[a:b] )$ is a random variable such that $0 \leq W( A, \YY'[a:b] ) \leq K$.
Taking the variance in \eqref{eq:N-counts-randomized} then gives
\begin{align*}
\Var [ N(A, \YY'[a:b])]   & \leq 2 \Var [ N(A, \YY[a + K:b]) ] + 2 \Var [  W( A, \YY'[a:b] ) ]\\
                        & \leq 2 C |b - a - K + 1| + 2 K^2
                         \leq (2 C  + 2 K^2)|b-a+1|,
\end{align*}
where in the second inequality we used the variance condition \eqref{eq:VarCondition} for the process $\YY$ together with
the almost sure bound on $W( A, \YY'[a:b] )$.
The last inequality confirms the  variance condition \eqref{eq:VarCondition} for the  process $\YY'$.

To check the local uniformity condition  \eqref{eq:ConditionallyUniform}
for $\YY'$, we take integers $a\leq b$, and we fix
a Borel set $A \subseteq Q \in \QQ(\alpha).$
By the $(\alpha, M)$ local uniformity of $\YY$
we know that there is a set
$
\S \subseteq \{Y_t: Y_t \in A \text{ and } t \in [a+K:b] \}
$
such that
\begin{equation}\label{eq:S-cardinality-RandomShift}
N(A, \YY[a + K:b]) - M  \leq | \S | \leq N(A, \YY[a + K:b])
\end{equation}
and
$\S \stackrel{\rm dpp}{=} \{U_s(A): 1 \leq s \leq |\S| \}$.

We now just need to check we have the local uniformity condition \eqref{eq:ConditionallyUniform}
for $\YY'$ if we take $\S'=\S$. With this choice, the bounds $0\leq I < K$ give us the inclusion
\begin{align*}
\S' & \subseteq \{Y_t: Y_t \in A \text{ and } t \in [a+K:b] \} \\
    & \subseteq \{Y_t: Y_t \in A \text{ and } t \in [a+I:b+I] \}
    = \{Y'_t: Y'_t \in A \text{ and } t \in [a:b] \}.
\end{align*}
The set $\S'$ inherits the distributional requirement of \eqref{eq:ConditionallyUniform} from $\S$, so we only need
to check its cardinality requirements.
From  \eqref{eq:N-counts-randomized} and \eqref{eq:S-cardinality-RandomShift}
we get the upper bound
$$
|\S' | = |\S| \leq  N(A, \YY[a + K:b]) \leq  N(A, \YY'[a:b]),
$$
as well as the lower bound
$$
 N(A, \YY'[a:b]) - K - M \leq N(A, \YY[a + K:b])  - M \leq |\S | = |\S'|,
$$
so the cardinality requirements of \eqref{eq:ConditionallyUniform} also hold.
\end{proof}

A $T_{\epsilon, N}$ transformation is just a $H_{\epsilon, N}$ transformation followed by an index shift by an independent
random variable $I$ that has the uniform distribution on the set of indices $[0:2N-1]$.
Proposition \ref{pr:local-unif} and Lemma \ref{lm:YTransformation} tell us that each of these actions
preserves local uniformity. These observations
give us the final lemma of the section, where, for simplicity, we take a generous
value for the size $\widetilde{M}$ of the exception set.

\begin{lemma}[Local Uniformity and $T_{\epsilon, N}$ Transformations]\label{lm:local-unif-T-transformation}
If $\XX$ is an $(\alpha,M)$ locally uniform process
for some $0< \alpha \leq 1$ and $M \in \N$, then for each $0 < \epsilon < \alpha$ and $N < \infty$ the process
$$
\widetilde \XX = T_{\epsilon, N} (\XX),
$$
is an $(\widetilde{\alpha}, \widetilde{M})$ locally uniform process where
$$
0 < \widetilde{\alpha} < \min\{ \epsilon, \alpha - \epsilon\} \quad \text{ and } \quad \widetilde{M}= M +6N.
$$

\end{lemma}

\section{Iterated $T_{\epsilon, N}$ Transformations and a Limit Process}\label{se:IteratedTransformations}

We now consider the construction of a process $\XX^* =( X^*_t: t \in \Z)$
that we obtain as a limit of iterated $T_{\epsilon, N}$ transformations.
First we fix a sequence of integers
$1\leq N_1 < N_2 < \cdots $ and a sequence of real numbers that we write as
$1 > \epsilon_1 > \epsilon_2 > \cdots >0$.
Next, we let $\XX^{(0)}=(X^{(0)}_t: t \in \Z)$ be the doubly infinite sequence of independent random variables with the
uniform distribution on $\T$,  and we consider the
infinite sequence of stationary processes $\XX^{(1)}, \XX^{(2)}, \XX^{(3)},  \ldots$
that one obtains by successive applications of appropriate $T_{\epsilon, N}$ transformations:
\begin{equation}\label{eq:GrandRecursion}
\XX^{(1)} = T_{\epsilon_1, N_1} ( \XX^{(0)} ), \quad
\XX^{(2)} = T_{\epsilon_2, N_2} ( \XX^{(1)} ), \quad
\XX^{(3)} = T_{\epsilon_3, N_3} ( \XX^{(2)} ), \quad
\cdots.
\end{equation}

We now let $\T^\infty$ be the set of doubly infinite sequences $\x=(\ldots, x_{-1}, x_0, x_1, \ldots)$
where $x_t \in \T$ for each $t \in \Z$,
and we view $\T^\infty$ as a topological space with respect to the product topology.
By $\BB(\T^\infty)$ we denote the $\sigma$-field of Borel sets of $\T^\infty$,
and we let $\BB(\T^{[-m:m]})$ denote the smallest sub-$\sigma$-field of
$\BB(\T^\infty)$ such that each continuous function $f:\T^\infty \mapsto \R$ of the form
$\x \mapsto f(x_{-m}, \ldots, x_{-1}, x_0, x_1, \ldots,  x_m)$ is $\BB(\T^{[-m:m]})$ measurable.
In less formal language,  $\BB(\T^{[-m:m]})$ is the subset of elements of $\BB(\T^\infty)$
that do not depend on $x_t$ for $|t| > m$.

Next, we take  $\M$ to be the set of Borel probability measures on $\T^{\infty}$,
and we note that $\M$ is a complete metric space
if one defines the distance $\rho(\mu, \mu')$
between the Borel measures $\mu$ and $\mu'$ by setting
\begin{equation}\label{eq:DefRho}
\rho(\mu, \mu')= \sum_{m=1}^\infty 2^{-m} \sup \{ \, |\mu(\A) -\mu'(\A)|: \, \A \in \BB(\T^{[-m:m]}) \, \}.
\end{equation}
To show that the
sequence $\XX^{(1)}, \XX^{(2)}, \XX^{(3)},  \ldots$
converges in distribution to a process  $\XX^{*}$,  it suffices to show that
if we define the measures $\mu_1, \mu_2,  \ldots$ on $\BB(\T^\infty)$ by setting
\begin{equation}\label{eq:DefMuJ}
\mu_j(\A) = P (\XX^{(j)} \in \A), \quad \quad \text{for each } \A \in \BB(\T^\infty),
\end{equation}
then the sequence  $\mu_1, \mu_2,  \ldots$
is a Cauchy sequence under the metric $\rho$.
Fortunately, the Cauchy criterion can be verified under a mild condition on
the defining sequence of integers $N_1, N_2, \ldots $.

\begin{lemma}[A Condition for Convergence]\label{lm:GeneralConvergence}
If the processes $\XX^{(1)}, \XX^{(2)}, \XX^{(3)},  \ldots$  are defined by the
iterated $T_{\epsilon, N}$ transformations \eqref{eq:GrandRecursion} and if
\begin{equation*}\label{eq:ConvergenceConditionOnN}
\sum_{j=1}^\infty \frac{1}{N_j} < \infty,
\end{equation*}
then the sequence of processes $\XX^{(1)}, \XX^{(2)}, \XX^{(3)},  \ldots$
converges in distribution to a stationary,  translation invariant
process $\XX^*= (\ldots, X^*_{-1}, X^*_{0}, X^*_{1}, \ldots )$ such that $X^*_{1}$ has the uniform distributed on
$\T$.
\end{lemma}

\begin{proof}
By the closeness inequality \eqref{eq:ClosenessIneq} and the definition
\eqref{eq:DefMuJ} of $\mu_j$, we have for all $m=1,2, ...$ that
\begin{equation*}\label{eq:ClosenssIneqSym}
\sup \{ |\mu_j(\A) -\mu_{j+1}(\A)|: \, \A \in \BB(\T^{[-m:m]}) \}\leq \frac{ 2 m }{N_{j+1}}.
\end{equation*}
The definition \eqref{eq:DefRho} of the metric $\rho$ and a simple summation then give us
$$
\rho(\mu_j, \mu_{j+1}) \leq \frac{4}{N_{j+1}},
$$
so, by the completeness of  the metric space $(\M, \rho)$,
the sequence of processes $\XX^{(1)}, \XX^{(2)}, \XX^{(3)},  \ldots$ converges in distribution to
a process $\XX^*$.
By Lemmas \ref{lm:PassageFromPeriodicity}, \ref{lm:UniformPreservation}, and \ref{lm:translation-invariance}, we know that
each of the processes  $\XX^{(j)}$ is stationary and translation invariant. Moreover, each of these processes
has uniform marginal distributions.
The process $\XX^*$ inherits all of these properties through convergence in distribution.
\end{proof}

\section{Path Lengths for the Limit Process}\label{se:Pathlengths}

The next lemma expresses a kind of Lipschitz property for the shortest path functional.
Specifically, it bounds the absolute difference in the expected value of
$L(Z)$ and $L(\widetilde Z)$, where $Z$ and $\widetilde Z$ are arbitrary $n$-dimensional random vectors with
values in $\T^n$ and where for $(z_1,z_2, \ldots, z_n) \in \T^n$, we write
$L(z_1,z_2, \ldots, z_n)$ for the length of the shortest path through the points $z_1,z_2, \ldots, z_n$.

The lemma is stated and proved for general $Z$ and $\widetilde Z$, but our typical choice will be
$Z = \XX[0: n-1]$ and $\widetilde Z = \widetilde \XX[0:n-1]$.
We also recall that if $\BB(\T^n)$ denotes the set of all Borel subsets of $\T^n$,
then the \emph{total variation distance} between $Z$ and $\widetilde Z$
is given by
$$
d_{\rm TV}(Z, \widetilde  Z)=\sup \{ |\P(Z \in \A) - \P(\widetilde  Z \in \A) \,|: \A \in \BB(\T^n)\}.
$$
We also recall that the function
$(z_1, z_2, \ldots, z_{n}) \mapsto n^{-1/2} L(z_1, z_2, \ldots, z_{n})$ is uniformly bounded; in fact,
by early work of \citeasnoun{Few:MATH1955}, it is bounded by $3$.

\begin{lemma}\label{lm:BHHRobustness}
For all random vectors $Z$ and $\widetilde  Z$ with values in $\T^{n}$ we have
\begin{equation*}
|\E [L(Z)] - \E[L(\widetilde  Z)]| \leq 3 n^{1/2} d_{\rm TV}(Z, \widetilde  Z).
\end{equation*}
\end{lemma}
\begin{proof}
By the maximal coupling theorem \cite[Theorem 5.2]{Lin:DOVER2002} there exists
a probability space and a random pair $(Z', \widetilde Z')$ such that
$Z' \stackrel{d}{=}Z$, $\widetilde Z' \stackrel{d}{=}\widetilde  Z$ and
$$
\P(Z'\not=\widetilde Z')= d_{\rm TV}(Z, \widetilde  Z).
$$
Now, if we set
$
L^*_n = \max\{ L(z_1, z_2, \ldots, z_n): z_t \in \T \text{ and }  t \in [1:n] \},
$
then we have
\begin{align*}
|\E[L(Z)] - \E[L(\widetilde  Z)]|&=| \E[L(Z')] - \E[L(\widetilde Z')]| \\
& \leq \P(Z'\not=\widetilde Z')L^*_n  \\
& \leq  3 n^{1/2} d_{\rm TV}(Z, \widetilde  Z),
\end{align*}
where in the last line we used the classic bound $L^*_n \leq 3 n^{1/2}$ from \citeasnoun{Few:MATH1955}.
\end{proof}

The immediate benefit of Lemma \ref{lm:BHHRobustness} is that it gives us a way to estimate the cost of
a minimal path through the points $X^*_0, X^*_1, \ldots, X^*_{n-1}$.

\begin{lemma}[Shortest Path Differences in the Limit]\label{lm:TSPdifferences}
For all $0 \leq j < \infty $ and all $n \geq 1$ we have
\begin{equation}\label{eq:TSPofXstar}
|\E [ L(\XX^{(j)}[0:n-1]) ] - \E [ L(\XX^{*}[0:n-1]) ] | \leq 3 \,  n^{3/2} \sum_{k=1}^\infty \frac{1}{N_{j+k}}
\end{equation}
\end{lemma}

\begin{proof}
Using the shorthand $\XX^{(j+k)}_n = \XX^{(j+k)}[0:n-1]$ for all $k \in \{0,1,2,\ldots\}$, one has by
the triangle inequality that
\begin{equation}\label{eq:TSP-difference-triangle-ineq}
|\E [ L(\XX^{(j)}_n) ] - \E [ L(\XX^{*}_n) ] |
\leq \sum_{k=1}^\infty |\E [ L(\XX^{(j+k)}_n) ] - \E [ L(\XX^{(j+k-1)}_n) ] |.
\end{equation}
Lemma \ref{lm:BHHRobustness} then tells us that
\begin{equation}\label{eq:TSP-difference-BHH-robust}
|\E [ L(\XX^{(j+k)}_n) ] - \E [ L(\XX^{(j+k-1)}_n) ] |
\leq 3 \, n^{1/2} d_{\rm TV}(\XX^{(j+k)}_n, \XX^{(j+k-1)}_n),
\end{equation}
and Lemma \ref{lm:CloseInDistribution} implies that
\begin{equation}\label{eq:TSP-difference-BHH-close}
d_{\rm TV}(\XX^{(j+k)}_n, \XX^{(j+k-1)}_n) \leq  \frac{n}{N_{j+k}},
\end{equation}
so using \eqref{eq:TSP-difference-BHH-robust}
and \eqref{eq:TSP-difference-BHH-close} in the sum \eqref{eq:TSP-difference-triangle-ineq}
completes the proof of the lemma.
\end{proof}

\section{Parameter Choices}

To pass from the general iterative construction \eqref{eq:GrandRecursion}
to the process required by Theorem \ref{th:noBHHThmExpectations},
we need to make parameter choices that go beyond those required by Lemma \ref{lm:GeneralConvergence}
on a sufficient condition for convergence.

First we fix a sequence $\eta_1, \eta_2, \ldots$  of values in $(0,1)$ that decrease
monotonically to zero as $j \rightarrow \infty$; these values just serve to provide us with a measure
of smallness of scale. We then inductively define the values $N_j$ and $\epsilon_j$ through which we finally
define $\XX^*$ by the sequence  \eqref{eq:GrandRecursion} of transformations
$T_{\epsilon_j, N_j}$, $j=1,2, \ldots$. To begin the construction, we can take any $\epsilon_1 \in (0,1)$ and any integer $N_1\geq 2$.
Subsequent values with $j \geq 2$ are determined by two rules:

\begin{enumerate}[align=left,labelindent=!,labelwidth=!,labelsep=0pt]
    \item[\eqnum\label{eq:rule1}] \textsc{Rule 1.}
            We choose an integer $N_{j}$ such that $N_j > j^2 N_{j-1}$ and such that
            \begin{equation*}
                    |\, \E[L(\XX^{(j-1)}[0:n-1])] - \beta n^{1/2} \,| \leq \eta_j n^{1/2}
                    \quad \quad \text{for all } n \geq \lfloor j^{-1} N_{j} \rfloor.
            \end{equation*}

    \item[\eqnum\label{eq:rule2}] {\sc Rule 2.}
        We choose an $\epsilon_j \in (0, \epsilon_{j-1})$ such that
        \begin{equation*}
        \epsilon_j j^{1/2} N_j^{1/2} \leq \eta_j.
        \end{equation*}
\end{enumerate}

%
%

The first rule leans on the fact that each of the processes $\XX^{(j)}$, $ 1 \leq j < \infty,$
is locally uniform so, from Proposition \ref{pr:BHH4LocalUnif}, we have that
$$
n^{-1/2} \, \E[L(\XX^{(j-1)}[0:n-1])]  \rightarrow \beta \quad \quad \text{as } n \rightarrow \infty.
$$
This guarantees the existence of the integer $N_j$ required by \eqref{eq:rule1},
and, once $N_j$ is determined, it is trivial to choose $\epsilon_j$ to satisfy the second rule.

\section{Estimation of Expected Path Lengths: Proof of Theorem \ref{th:noBHHThmExpectations}}\label{se:ProofofNoBHHnew}

The $j$'th stage of the construction \eqref{eq:GrandRecursion} takes a doubly infinite sequence
$\XX^{(j-1)}$ to another doubly infinite sequence $\XX^{(j)}$ by a
$T_{\epsilon_j, N_j}$ transformation that we define in two steps.
Specifically, if $I$ is a random variable with the uniform distribution
on $\{ 0, 1, \ldots, 2N_j-1\}$ that is independent of $\XX^{(j-1)}$, then we have
$$
\widehat{\XX}^{(j-1)} =
H_{\epsilon_j, N_j}(\XX^{(j-1)})
\quad \text{and} \quad
X^{(j)}_t = \widehat X^{(j-1)}_{t+I}\quad \text{for all } t \in \Z.
$$
We now claim that we have the bounds
\begin{align}\label{eq:pathEst1}
L( \XX^{(j)}[0:2jN_j-1] ) &\leq L ( \widehat{\XX}^{(j-1)}[0:2(j+1)N_j-1] )\\
& \leq L ( \XX^{(j-1)}[0:(j+1)N_j-1] ) + 2 \epsilon_j (j+1) N_j.  \notag
\end{align}
To check the first inequality, we recall that $0 \leq I < 2N_j-1$ so the set
$$
\WW=\{ \, \widehat{{X}}^{(j-1)}_t :  0 \leq t \leq 2(j+1)N_j-1 \, \}
$$
is a superset of
$
\{ \widehat{{X}}^{(j-1)}_t:  I \leq t  \leq I + 2 j N_j-1 \}=\{ X^{(j)}_t :  0 \leq t \leq 2 j N_j-1 \} .
$

To check the second inequality of \eqref{eq:pathEst1} takes more work.
We first note that the segment $ \widehat{\XX}^{(j-1)}[0:2(j+1)N_j-1]$
contains only complete blocks, so one also has
$$
\WW = \{ X^{(j-1)}_t :  0 \leq t \leq (j+1)N_j -1 \} \cup \{ X^{(j-1)}_t(\epsilon) :  0 \leq t \leq (j+1)N_j -1 \},
$$
and, to prove the second inequality of \eqref{eq:pathEst1},
we construct a suboptimal path through the points of $\WW$.

To build this path, we just take
an optimal path $\PP$ through the set of points
$ \{ X^{(j-1)}_t :  0 \leq t \leq (j+1)N_j -1 \}$
and then for each $t \in [0:(j+1)N_j-1]$ we adjoin to $\PP$ the loop that goes from
${X}^{(j-1)}_t$ to its shifted twin  $X^{(j-1)}_t (\epsilon_j)$  and back to ${X}^{(j-1)}_t$.
The suboptimal path has  $(j+1)N_j$ loops and each of these has length $2 \epsilon_j$, so
the suboptimal path through $\WW$ has a cost that is
bounded by the last sum in \eqref{eq:pathEst1}.

Now, when we take expectations in \eqref{eq:pathEst1}, the rules \eqref{eq:rule1} and \eqref{eq:rule2} give us
\begin{align}
(2jN_j)^{-1/2}\E [ L(\XX^{(j)}[0 \! : \! 2jN_j \! - \! 1] ) ] &\leq \! (\beta \! + \! \eta_j)
(1 \! + \! j^{-1/2}) 2^{-1/2} \! + \! 2 \epsilon_j (j^{1/2} \! + \! j^{-1/2})  N_j^{1/2} \notag \\
&\leq \! (\beta \! +  \! \eta_j)(1 \! + \! j^{-1/2}) 2^{-1/2} \! + \! 4 \eta_j. \notag
\end{align}
This bound together with \eqref{eq:TSPofXstar} for  $n=2jN_j$ then gives us
\begin{align}
(2jN_j)^{-1/2}\E [L (\XX^{*}[0:2jN_j-1])]
& \leq  \! (\beta \! + \! \eta_j)
(1 \! + \! j^{-1/2})  2^{-1/2} \! +  \! 4 \eta_j  \! + \!  6 jN_j \sum_{k=1}^\infty \frac{1}{N_{j+k}}\notag \\
&\leq  \! (\beta \! + \! \eta_j)(1 \! + \! j^{-1/2})  2^{-1/2} \! +  \! 4 \eta_j  \! + \!  6/j, \nonumber
\end{align}
where, in the second inequality, we estimate the sum using the
strict inequality $j^{2k} N_j < N_{j+k}$ which holds for all $k\geq 1$ and $j \geq 2$ by
the first part of our first parameter formation rule \eqref{eq:rule1}.
This last displayed bound is more than one needs to complete the proof of the first inequality \eqref{noBBBinequaliesMean}
of Theorem \ref{th:noBHHThmExpectations}.

The second inequality \eqref{noBBBinequaliesMean} of Theorem \ref{th:noBHHThmExpectations} is easier. If we take
$n= \lfloor j^{-1} N_j \rfloor$ in the estimate \eqref{eq:rule1} given by our first parameter rule, then we have
$$
\beta -\eta_j \leq \lfloor j^{-1} N_j \rfloor^{-1/2} \E[L(\XX^{(j-1)}[0:\lfloor j^{-1} N_j \rfloor-1])].
$$
Now, if we use the bound \eqref{eq:TSPofXstar} for $n= \lfloor j^{-1} N_j \rfloor$ and if we estimate
the infinite sum exactly as we did before, then we have
$$
\beta -\eta_j  - 6/j \leq \lfloor j^{-1} N_j \rfloor^{-1/2} \E[L(\XX^{*}[0:\lfloor j^{-1} N_j \rfloor-1])],
$$
and this bound more than one needs to complete the proof of second inequality \eqref{noBBBinequaliesMean} of Theorem \ref{th:noBHHThmExpectations}.

\section{Theorem \ref{th:noBHHThmExpectations} Implies Theorem \ref{th:noBHHThm}}\label{se:ErgodicDecompTrick}

We now show that Theorem \ref{th:noBHHThm} follows from Theorem \ref{th:noBHHThmExpectations} and Choquet's representation theorem.
To set this up, we first consider the
set $\U \subset \M$ of Borel probability measures on $\T^\infty$ that are shift invariant and that have uniform marginals.
Here the shift transformation $\Theta: \T^\infty \rightarrow \T^\infty$ is defined by setting
$\Theta(\x)_t=x_{t+1}$ where  $\x=(\ldots, x_{-1}, x_0, x_1, \ldots) \in \T^\infty$, and a Borel measure $\nu$ on $\T^\infty$
is said to be shift invariant if
$\nu(\Theta^{-1}(\A))=\nu(\A)$ for each Borel set $\A \subseteq \T^\infty$.
Finally, to say that a measure $\nu$ on $\T^\infty$ has uniform marginals just means that for each Borel set
$ A \subseteq \T$ and for each fixed $t \in \Z$ we have
$\nu(\A_t) = \lambda(A)$ when we set $\A_t=\{\x: x_t  \in A  \}$
and where $\lambda(A)$ denotes the Lebesgue measure of the set $A \subseteq \T$.

Under the product topology the space  $\T^\infty$ is compact, and, under the topology of weak convergence,
the space $\M$ of probability measures on $\T^\infty$ is metrizable, locally convex, and compact.
The set $\U \subset \M$ of shift invariant measures with uniform marginals is trivially convex,
and $\U$ is a closed subset of $\M$. Hence, under the topology of weak convergence, the set
$\U$ is also compact.

We now consider the process $\XX^*$ given by Theorem \ref{th:noBHHThmExpectations}, and we define a
measure $\mu$ on $\T^\infty$ by setting
\begin{equation}\label{eq:MuDef}
\mu(\A)= P(\XX^* \in \A) \quad \text{for each Borel set } \A \subseteq \T^\infty.
\end{equation}
Stationarity of  $\XX^*$ tells us that
$\mu$ is invariant under the shift transformation $\Theta$, and, since $\mu$ inherits the uniform marginal property from $\XX^*$,
we have $\mu \in \U$.

Next, we denote the symmetric difference between the sets $\A$ and $\A'$ by $\A \Delta \A'$, and we recall that
a probability measure $\nu$ on $\T^\infty$ is \emph{ergodic} if for any Borel $\A \subseteq \T^\infty$, one has
$$
\nu(\A  \Delta  \Theta^{-1}(\A)) =0 \quad \text{if and only if}\quad \nu(\A)=0 \, \, \text{or } \nu(\A)=1.
$$
Moreover, if we let $\U_e$ denote the set of
of extreme points of the convex set $\U$, then it is a very pleasing fact that $\U$ is exactly equal to the set of ergodic measures in $\U$,
\citeaffixed[Example 8.17, pp.~129--131.]{Simon2011}{cf.}

Choquet's representation theorem
\citeaffixed[Theorem 10.7]{Simon2011}{cf.} now tells us that there is a
probability measure $D_\mu$ with support on $\U_e \subset \U$ such that
\begin{equation}\label{eq:ChoquetRepr}
\mu(\A) = \int_{\U_e} \nu(\A) \, D_\mu (d \nu) \quad\quad \text{ for each Borel } \A \subseteq \T^\infty.
\end{equation}
In other words, any shift invariant probability measure on $\T^\infty$ with uniform marginal distributions is an average
of ergodic, shift invariant probability measures on $\T^\infty$ that have uniform marginal distributions.
This representation will lead us quickly to Theorem \ref{th:noBHHThm}.

For an $\x=(\ldots, x_{-1}, x_0, x_1, \ldots) \in \T^\infty$
we write the $[1:n]$-segment of $\x$
with the shorthand $\x_n = (x_1,x_2, \ldots, x_n)$, and, for constants $c_1<c_2$,
we consider the set
\begin{equation*}\label{DefDelta}
\Gamma[c_1, c_2]
= \big\{  \x \in \T^\infty \! : \liminf_{n \rightarrow \infty}  n^{-1/2} L(\x_n)  \leq  c_1  \text{ and }
                              c_2 \leq  \limsup_{n \rightarrow \infty}  n^{-1/2}L(\x_n) \big\}.
\end{equation*}
If $\mu(\Gamma[c_1,c_2])=0$ for all $c_1< c_2$, then  $n^{-1/2} L(\x_n)$ converges with
$\mu$-probability one. Since this ratio is bounded, the dominated convergence theorem implies the convergence of the expectations
$ n^{-1/2} \E_\mu [ L(\x_n) ]=n^{-1/2} \E[ L(\XX^{*}[1:n]) ]$, where in the last equality we just used the definition \eqref{eq:MuDef} of $\mu$.
By Theorem \ref{th:noBHHThmExpectations} we know that we do not have the convergence of these expectations, so there must be some
pair $c_1< c_2$ for which we have  $0 <\mu(\Gamma[c_1,c_2])$. Finally, by the representation \eqref{eq:ChoquetRepr} we have
$$
0<\mu(\Gamma[c_1,c_2]) = \int_{\U_e} \nu(\Gamma[c_1,c_2]) \, D_\mu (d\nu),
$$
so there is some ergodic measure $\nu \in \U_e$ for which we have $0< \nu(\Gamma[c_1,c_2])$.

The sets $\Gamma[c_1,c_2]$ and $\Theta^{-1}(\Gamma[c_1,c_2])$
are identical, so $\Gamma[c_1,c_2]$ is an invariant set for $\nu$, or any other measure. Since $0< \nu(\Gamma[c_1,c_2])$, the
ergodicity of $\nu \in \U_e$ gives us $\nu(\Gamma[c_1,c_2])=1$.
Finally, if we take $\XX$ to be the stationary process determined by the shift transformation $\Theta$ and the ergodic measure $\nu$, then
we have
\begin{equation}\label{eq:XXDef}
P(\XX \in \A)=\nu(\A) \quad \text{for each Borel set } \A \subseteq \T^\infty.
\end{equation}
By construction, the process $\XX$ is stationary and ergodic, and $\XX$ also inherits from $\nu$
the property of uniform marginal distributions. Finally, by \eqref{eq:XXDef} and the
definition of $\Gamma[c_1,c_2]$, we see that $\XX$ has all of the features required by Theorem \ref{th:noBHHThm}.

\begin{remark}\harvardparenthesis{none}
Here, instead of using Choquet's theorem, one could consider using the
ergodic decomposition theorems of Krylov and Bogolioubov
(cf. \citeasnoun{Dyn:AP1978} or \citeasnoun[Theorem 10.26, p.~196]{Kal:SPRI2002}),
but, since the usual statements of these theorems
to not immediately accommodate the restriction to measures with uniform marginal distributions, it
is simpler to work directly with Choquet's theorem where no modifications are required. \harvardparenthesis{round}

For a full treatment of Choquet's theorem one can consult \citeasnoun{Phelps2001} or \citeasnoun{Simon2011},
but for the existence theorem in the metrizable case one can appeal more easily to the short proof of \citeasnoun{Bonsall1963}
which uses little more than the Hahn-Banach theorem and the Riesz representation theorem.
\end{remark}

\section{Extensions, Refinements, and Problems}\label{se:conclusions}

There are easily proved analogs of
Theorems \ref{th:noBHHThm} and \ref{th:noBHHThmExpectations} for many of the functionals of combinatorial optimization for which
one has the analog of the Beardwood-Halton-Hammersley theorem.
In particular, one can show that the analogs of
Theorems \ref{th:noBHHThm} and \ref{th:noBHHThmExpectations} hold for the
minimal spanning tree (MST) problem studied in \citeasnoun{Ste:AP1988}
and for the minimal matching problem studied in \citeasnoun{Rhee:AAP1993}.
In these cases, the construction of the processes in Theorems \ref{th:noBHHThm} and \ref{th:noBHHThmExpectations} needs almost no alteration.
The main issue is that one needs to establish a proper analog of Proposition \ref{pr:BHH4LocalUnif}, but this is often easy. Once an analog of
Proposition \ref{pr:BHH4LocalUnif} is in hand,
one only needs to make
few cosmetic changes to the arguments of Section \ref{se:ProofofNoBHHnew}.

Still, there are interesting functionals for which it is not as clear how one can adapt the proofs of
Theorems \ref{th:noBHHThm} and \ref{th:noBHHThmExpectations}.
One engaging example is the sum of the edge lengths of the Voronoi tessellation. In this case, the analog of the
BHH theorem was developed by \citeasnoun{Mil:MB1970} for Poisson sample sizes, and later by
\citeasnoun{McGivYuk:SPA1999} for fixed sample sizes (and with complete convergence).
A second, much different example, is the
length of the path that one obtains by running the Karp-Held algorithm for the TSP.
The expressly algorithmic nature of this functional introduces several new twists, but, nevertheless,
\citeasnoun{GoeBer:MOR1991}
obtained the analog of the BHH theorem.

These two functionals are ``less local'' than the TSP, MST, or minimal matching functionals; in particular, they are not as
amenable to suboptimal patching bounds such as those we used in Section \ref{se:ProofofNoBHHnew}.
Nevertheless, these functionals are sufficiently local to allow for analogs of the BHH theorem, so it seems
probable that the natural analogs of Theorems \ref{th:noBHHThm} and \ref{th:noBHHThmExpectations} would hold as well.

There are two further points worth noting. First, at the cost of using more
complicated versions of the  transformations $H_{\epsilon, N}$ and $T_{\epsilon, N}$, one can replace the
infimum bound $2^{-1/2}\beta$ of Theorem \ref{th:noBHHThmExpectations} with a smaller constant.
Since the method of Section \ref{se:ErgodicDecompTrick} shows that any
infimum bound less that $\beta$ suffices to prove Theorem \ref{th:noBHHThm},
we did not pursue the issue of a minimal infimum bound.

Finally, it is feasible that the process  $(X^*_t: t \in \Z )$ that was constructed for the proof of
Theorem \ref{th:noBHHThmExpectations} is itself ergodic ---  or even mixing.  If this could be established, then
one would not need the Choquet's integral representation argument of Section \ref{se:ErgodicDecompTrick}.
Unfortunately, it does not seem easy to prove that the process $(X^*_t: t \in \Z )$ is ergodic, even though this may be somewhat intuitive.

\section*{Acknowledgments}

The authors are grateful to the editors and the anonymous readers for their careful comments that went beyond
the normal. Their generous help has been of major benefit.

\end{document}